\documentclass[12pt]{article}
\usepackage[backend=biber,
sorting=none]{biblatex}  
\usepackage[utf8]{inputenc}
\usepackage{amsfonts}
\usepackage{hyperref}
\hypersetup{
pdftitle={An edge-type state integral over local field and the A polynomial},
pdfsubject={Quantum Topology, Mathematical Physics},
pdfauthor={Honghuai Fang},
pdfkeywords={Teichmüller TQFT, A-polynomials}
}
\usepackage[justification=centering]{caption}
\usepackage{amsmath}
\usepackage{xcolor}
\usepackage{amsthm}
\usepackage{pdflscape}
\usepackage{pgfplots}
\usepackage{mathrsfs}
\usepackage{amssymb}
\usepackage{float}
\usepackage{makecell}
\usepackage{enumitem,fancyvrb}
\usepackage{booktabs,multirow,longtable,makecell}

\usepackage{pdfpages}
\usepackage{amsmath}
\allowdisplaybreaks[4]
\usepackage{mathrsfs}
\usepackage{amsfonts}
\usepackage{dsfont}
\usepackage{bm}
\usepackage{bbm}
\usepackage{comment}
\usepackage{graphicx}
\usepackage{multirow}
\usepackage{indentfirst}
\usepackage{arydshln}
\usepackage{shorttoc}
\usepackage{float}
\usepackage{listings}
\usepackage{diagbox}
\usepackage{hyperref}

\usepackage{threeparttable}
\graphicspath{{img/}}

\newtheorem{theorem}{Theorem}[section]

\newtheorem{remark}{Remark}[section]

\theoremstyle{remark}
\newtheorem{example}{Example}[section]

\let\oldthebibliography\thebibliography
\renewcommand\thebibliography[1]{%
	\oldthebibliography{#1}%
	\setlength{\parskip}{\bibitemsep}%
	\setlength{\itemsep}{\bibparskip}%
}

\def\da{\dot{a}}
\def\dda{\ddot{a}}

\def\ddb{\ddot{b}}
\def\dc{\dot{c}}
\def\ddc{\ddot{c}}

\def\ddl{\ddot{\lambda}}

\def\ddm{\ddot{\mu}}

\numberwithin{equation}{section}

\title{An edge-type state integral over local field and A-polynomials}

\author{HONGHUAI FANG}

\date{}

\begin{document}

\maketitle

\begin{abstract}
To each local field, Garoufalidis and Kashaev recently associate a quantum dilogarithm that satisfies a pentagon identity and some symmetries. By employing an angled version of these quantum dilogarithms, they developed three generalized TQFTs, one given by a face state-integral and others by edge state integrals. These TQFTs produce distributional invariants for one-cusped three-manifolds, which are believed to be related to counting points on the A-polynomial curve. In this paper, we will calculate partition functions of an edge-type generalized TQFT over a local field for several examples and prove the appearance of A-polynomial in these new invariants.
\end{abstract}

\tableofcontents

\section{Introduction}
Analyzing quantum Chern-Simons theory\cite{witten89} gets a lot more complex and is less advanced when the structure group $G$ is not compact, unlike when $G$ is compact. Andersen and Kashaev introduced a new theory called Teichmüller TQFT\cite{Andersen2014ATF}, where the theory's operators work on infinite-dimensional vector
spaces. This version of TQFT has a partition function that creates a quantum invariant based on quantum Teichmüller theory. They built this TQFT using positivity on shape structure and calculations on distributions related to Faddeev's quantum dilogarithm\cite{faddeev1994quantum}.

Using the Weil transform of quantum dilogarithm, Anderson and Kashaev \cite{andersen2013new} suggested an edge-type state-integral model for the Teichmüller TQFT. In this model, the circle-valued state variables are associated with the edges of oriented shaped triangulations. The partition functions can be analytically continued to arbitrary complex shapes, liberating the theory from needing positive shape constraints. Also, like the Turaev-Viro model\cite{Kashaev2012ATO} and 3D index model\cite{Garoufalidis2017AME}, there are no restrictions on the topology of the cobordisms.

Recently, Garoufalidis and Kashaev\cite{garoufalidis2023quantum} studied a generalized quantum dilogarithm over a local field that satisfies a pentagon identity and some symmetries. Using an angled version of these quantum dilogarithms, they developed three generalized TQFTs, one given by a face state-integral and others by edge state integrals. These TQFTs produce distributional invariants for one-cusped three-manifolds, which are believed to be related to counting points on the A-polynomial curve\cite{cooper1994plane}. Furthermore, the partition function of an edge-type generalized TQFT can be expressed as a multidimensional Barnes-Mellin integral or a period on an A-polynomial curve, similar to what has been observed in mirror symmetry of Calabi-Yau manifolds.\cite{passare1996multiple}.

This paper will calculate partition functions of an edge-type generalized TQFT over local field $F$ defined in \cite{garoufalidis2023quantum} for several knot complements. We observe that they exhibit some exciting commonalities in their shapes as follows

\begin{equation}\label{the edge type partition function}
Z_F(S^3\backslash K)=\|\mathcal{F}_K(x,y)\|\delta_F(P_K(x,y)).  
\end{equation}
Furthermore, we will prove the following theorem.
\begin{theorem}
The function $P_K(x,y)$ coincides with the $\operatorname{PSL}(2,\mathbb{C})$ A-polynomial of knot $K$. 
\end{theorem}
See Section \ref{section 3} for more details.

\section{Teichmüller TQFT over local field}

\subsection{Quantum dilogarithm over local field}

 Denote by $\mathbb{T}$ the multiplicative group of complex numbers with absolute value 1. A topological group is called \textit{locally compact Abelian} (LCA) if the underlying topological space is locally compact, and Hausdorff and the underlying group is Abelian.
 For a LCA group $G$, its Pontryagin dual is the group $\widehat{G}$ of continuous group homomorphisms from $G$ to $\mathbb{T}$.

A \textit{Gaussian group} is a LCA group $\mathsf{A}$ equipped with a nondegenerate $\mathbb{T}$-valued quadratic form $\langle\cdot\rangle: \mathsf{A} \rightarrow \mathbb{T}$, i.e., a function that satisfies $\langle x\rangle=\langle-x\rangle$ for any $x \in \mathsf{A}$ and its polarization
$$
\langle x ; y\rangle:=\frac{\langle x+y\rangle}{\langle x\rangle\langle y\rangle}, \quad \forall(x, y) \in \mathsf{A}^2,
$$
is a non-degenerate bi-character. The form $\langle\cdot\rangle$ is called a \textit{Gaussian exponential} of $\mathsf{A}$, and the bi-character $\langle\cdot ; \cdot\rangle$ is called its \textit{Fourier kernel}.

A Gaussian group $\mathsf{A}$ has a canonically normalized Haar measure determined by the condition of the improper integral
$$
\int_{\mathsf{A}^2}\langle x ; y\rangle \mathrm{d}(x, y)=1 .
$$

A \textit{quantum dilogarithm} over a Gaussian group $\mathsf{A}$ is a tempered distribution over $\mathsf{A}$ $\varphi: \mathsf{A}\rightarrow \mathbb{C}$ that satisfies
\begin{enumerate}
    \item an \textit{inversion relation}: there exists a non-zero constant $c_{\varphi} \in \mathbb{C}^{\times}$such that
$$
\varphi(x) \varphi(-x)=c_{\varphi}\langle x\rangle
$$
for almost all $x \in \mathsf{A}$;
    \item a \textit{pentagon identity}:
$$
\varphi(x) \varphi(y)=\int_{\mathsf{A}}\langle x\rangle \mathrm{d}x  \int_{\mathsf{A}^3} \frac{\langle x-u ; y-w\rangle}{\langle u-v+w\rangle} \varphi(u) \varphi(v) \varphi(w) \mathrm{d}(u, v, w)
$$
for almost all pairs $(x, y) \in \mathsf{A}^2$.
\end{enumerate}

The pentagon identity can equivalently be written in the form of a distributional integral identity
$$
\tilde{\varphi}(x) \tilde{\varphi}(y)\langle x ; y\rangle=\int_{\mathsf{A}} \tilde{\varphi}(x-z) \tilde{\varphi}(z) \tilde{\varphi}(y-z)\langle z\rangle \mathrm{d} z,
$$
where
$$
\tilde{\varphi}(x):=\int_{\mathsf{A}} \frac{\varphi(y)}{\langle x ; y\rangle} \mathrm{d} y 
$$
is the inverse Fourier transformation of $\varphi$.

For example, for the Gaussian group $\mathbb{A}_N =\mathbb{R} \oplus \mathbb{Z} / N \mathbb{Z}$ with Gaussian exponential 
$$\langle(x, n)\rangle = e^{\pi i x^2} e^{-\pi i n(n+N) / N}$$
and Fourier kernel 
$$\langle(x,n);(y, m)\rangle = e^{2 \pi i x y} e^{-2 \pi i n m / N},$$
the level $N$ Faddeev's quantum dilogarithm $\mathrm{D}_{\mathrm{b}}(x, n)$ is a quantum dilogarithm\cite{andersen2014complex}. 

We recall some fundamental information about test functions and distributions on local fields, which are examples of LCA groups. We use $\mathcal{S}(F)$ to represent the Schwartz-Bruhat space of Fourier stable complex-valued test functions and $\mathcal{S}^{\prime}(F)$ to denote the dual space of tempered distributions. Note that the Fourier transform is an automorphism of the space $\mathcal{S}^{\prime}(F)$.

Fix a Haar measure $\mu_F$ on the additive group $(F,+)$ of a local field $F$ with the notation for the differential $\mathrm{d}_F x:=\mathrm{d} \mu_F(x)$. The multiplicative group $\mathsf{B}:=\left(F^{\times}, \times\right)$ is also a LCA group. We fix its Haar measure through the following relation for the differentials
$$
\mathrm{d}_{\mathsf{B}} x=\frac{\mathrm{d}_F x}{\|x\|}, \quad\|x\|:=\|x\|_F^d,
$$
where $d$ is the dimension of the field.

After Choosing the normalization of the Haar measure on $\hat{\mathsf{B}}$, the Dirac delta function $\delta_{\mathsf{B}}(x)$ is defined by 
$$
\delta_{\mathsf{B}}(x)=\int_{\hat{\mathsf{B}}} \alpha(x) \mathrm{d} \alpha.
$$

 Note that $\delta_{\mathsf{B}}(x)=0$ unless $x=1$. The relationship between delta functions $\delta_{\mathsf{B}}$ and $\delta_F$ on $\mathrm{B}$ and $F$ is
$$
\delta_{\mathsf{B}}(x)=\delta_F(x-1), \quad \forall x \in \mathsf{B} .
$$

We will denote elements of $\mathsf{B}$ by $x, y, \ldots$ and elements of $\hat{\mathsf{B}}$ by $\alpha, \beta, \ldots$ We will denote the canonical pairing $\hat{\mathsf{B}} \times \mathsf{B} \rightarrow \mathbb{T}$ by $(\alpha, x) \mapsto \alpha(x)$. The LCA groups $\mathsf{B}$ and $\hat{\mathsf{B}}$ can be combined to define a self-dual LCA group $\mathsf{A} =\hat{\mathsf{B}} \times \mathsf{B}$ which is a Gaussian group with Gaussian exponential
$$
\langle\cdot\rangle: \mathsf{A} \rightarrow \mathbb{T}, \quad\langle(\alpha, x)\rangle=\alpha(x)
$$
and the associated Fourier kernel
$$
\langle\cdot ; \cdot\rangle: \mathsf{A}^2 \rightarrow \mathbb{T}, \quad\langle(\alpha, x) ;(\beta, y)\rangle=\alpha(y) \beta(x).
$$
We define
$$
\varphi: \hat{\mathsf{B}} \times(\mathsf{B} \backslash\{-1\}) \subset \mathsf{A} \rightarrow \mathbb{T}, \quad \varphi(\alpha, x)=\alpha(1+x),
$$
which satisfies\cite{garoufalidis2023quantum} 
\begin{enumerate}
    \item the inversion relation: 
    $$\varphi(\alpha, x) \varphi(-\alpha, 1 / x)=\alpha(x);$$
    \item the pentagon identity:
    $$
\tilde{\varphi}(\alpha, x) \tilde{\varphi}(\beta, y) \alpha(y) \beta(x)=\int_{\mathsf{A}} \tilde{\varphi}(\beta-\gamma, y / z) \tilde{\varphi}(\gamma, z) \tilde{\varphi}(\alpha-\gamma, x / z) \gamma(z) \mathrm{d}(\gamma, z).
$$
\end{enumerate}
Thus, the function $\varphi$ is a quantum dilogarithm over $\mathsf{A}$.

We introduce an angled version $\Psi_{a, c}$ of the quantum dilogarithm $\varphi$, which satisfies some symmetry relations and an angle-dependent pentagon identity. The function $\Psi_{a, c}$ is the building block for the partition function of a tetrahedron. The relationship it satisfies will be used to show that the partition function of a triangulation is invariant under 2-3 Pachner moves, which makes it a topological invariant.

An angle $a=(\dot{a}, \ddot{a})$ is an element of $\mathsf{C}=\mathbb{R} \times \mathsf{B}$. Its conjugation is defined by
$$
\bar{a}:=\left(\dot{a},(\ddot{a})^{-1}\right) .
$$
When angles $a, b$, and $c$ are assigned to a tetrahedron, we assume that they satisfy
$$
a+b+c=\varpi:=(1,-1).
$$

For $a, c \in \mathsf{C}$, we define a function
$$
\Psi_{a, c}: \hat{\mathsf{B}} \times\left(\mathsf{B} \backslash\left\{\ddot{a}^{-1}\right\}\right) \rightarrow \mathbb{C}, \quad \Psi_{a, c}(\alpha, x)=\frac{1}{\alpha((1-\ddot{a} x) \ddot{c})} \frac{\|\ddot{a} x\|^{\dot{c}}}{\|1-\ddot{a} x\|^{1-\dot{a}}} ,
$$
 and denote by $\bar{\Psi}_{a, c}$ the function of $(\alpha, x) \in \hat{\mathsf{B}} \times(\mathsf{B} \backslash\{\ddot{a}\})$ defined by
\begin{equation}\label{general charged qd}
\bar{\Psi}_{a, c}(\alpha, x)=\overline{\Psi_{\bar{a}, \bar{c}}(\alpha, x)}=\alpha((1-x / \ddot{a}) / \ddot{c}) \frac{\|x / \ddot{a}\|^{\dot{c}}}{\|1-x / \ddot{a}\|^{1-\dot{a}}} .  
\end{equation}

The functions $\Psi_{a, c}$ and $\bar{\Psi}_{a, c}$ satisfy\cite{garoufalidis2023quantum}
\begin{enumerate}
    \item the symmetry relations:
    $$
\begin{aligned}
\Psi_{a, c}(-\alpha, 1 / x) \alpha(x) & =\bar{\Psi}_{a, b}(\alpha, x) ,\\
\int_{\hat{\mathsf{B}} \times \mathsf{B}} \Psi_{a, c}(\beta, y) \alpha(x / y) \beta(y / x) \mathrm{d}(\beta, y) & =\bar{\Psi}_{b, c}(\alpha, x);
\end{aligned}
$$
    \item the charged pentagon relation:
    \begin{equation}
         \begin{aligned}
        \Psi_{a_1,c_1}(\alpha, x) \Psi_{a_3,c_3}(\beta, y)=&\int_{\hat{\mathsf{B}} \times \mathsf{B}} \Psi_{a_0,c_0}(\alpha-\gamma, x / z) \Psi_{a_2,c_2}(\gamma, z) \Psi_{a_4,c_4}(\beta-\gamma, y / z)\\
        &\times \frac{\alpha(y / z) \beta(x / z)}{\gamma\left(x y / z^2\right)}\mathrm{d}(\gamma, z) 
        \end{aligned}   
    \end{equation}

provided that
$$
a_3=a_2+a_4, \quad c_3=a_0+c_4, \quad c_1=c_0+a_4, \quad a_1=a_0+a_2, \quad c_2=c_1+c_3 .
$$
\end{enumerate}

\begin{figure}[H]
    \centering
    \includegraphics[width=0.6\textwidth]{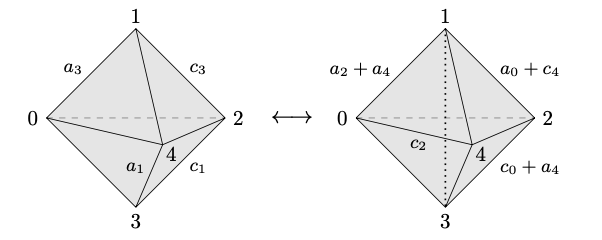}
    \caption{3-2 and 2-3 angled Pachner moves.}
    \label{img/angled pacher move}
\end{figure}

\subsection{Face-type Teichmüller TQFT over local field}
Consider the disjoint union of a finite number of a standard 3-simplex within $\mathbb{R}^3$. Each simplex is assumed to have ordered vertices, determining the orientation of its edges. Certain pairs of codimension one faces within this collection are glued together using an order-preserving, orientation-reversing affine homeomorphism. The quotient space yields a CW-complex $X$ called an oriented \textit{triangulated pseudo 3-manifold}. For $i \in\{0,1,2,3\}$, denote by $\Delta_i(X)$ the $i$-dimensional cell of $X$. For any $i>j$, define
$$
\Delta_i^j(X)=\left\{(a, b) \mid a \in \Delta_i(X), b \in \Delta_j(a)\right\} .
$$

The canonical boundary map
$$
\partial_i: \Delta_j(X) \rightarrow \Delta_{j-1}(X), \quad 0 \leq i \leq j
$$
is defined such that for a $j$-dimensional simplex $S=\left[v_0, v_1, \ldots, v_j\right]$ in $\mathbb{R}^3$ with ordered vertices $v_0, v_1, \ldots, v_j$,
$$
\partial_i S=\left[v_0, \ldots, v_{i-1}, v_{i+1}, \ldots, v_j\right], \quad i \in\{0, \ldots, j\}.
$$

For a tetrahedron $T=\left[v_0, v_1, v_2, v_3\right]$ with ordered vertices $v_0, v_1, v_2, v_3$, we define its sign by
$$
\operatorname{sign}(T)=\operatorname{sign}\left(\operatorname{det}\left(v_1-v_0, v_2-v_0, v_3-v_0\right)\right) .
$$

Let $X$ be an oriented triangulated pseudo 3-manifold. A $\mathsf{C}$-valued shape structure on $X$ is a function 
$$
\theta: \Delta_3^1(X) \rightarrow \mathsf{C}
$$
satisfies that the sum of angles at the edges emanating from each vertex of each tetrahedron equals $\varpi$.

We associate to $X$ the following \textit{kinematical kernal}
$$
K_X \in \mathcal{S}^{\prime}\left(\mathsf{A}^{\Delta_3(X)}\right), \quad K_X(z)=\int_{x \in \mathsf{A}^{\Delta_2(X)}}  \prod_{T \in \Delta_3(X)} K_T(x, z) \mathrm{d} x,
$$
where
$$
K_T(x, z):=\left\langle x_0 ; z(T)\right\rangle^{\operatorname{sgn}(T)} \delta_{\mathsf{A}}\left(x_0-x_1+x_2\right) \delta_{\mathsf{A}}\left(x_2-x_3+z(T)\right), \quad x_i:=x\left(\partial_i T\right).
$$

Also we associate to $X$ the following \textit{dynamical content}
$$
D_{X, \theta}(z) \in \mathcal{S}^{\prime}\left(\mathsf{A}^{\Delta_3(X)} \times \mathsf{C}^{\Delta_3^1(X)}\right), \quad D_{X, \theta}(z)=\prod_{T \in \Delta_3(X)} D_{T, \theta}(z(T))
$$
where
$$
D_{T, \theta}= \begin{cases}\Psi_{a, c} & \text { if } \operatorname{sgn}(T)=+1, \\ \bar{\Psi}_{a, c} & \text { if } \operatorname{sgn}(T)=-1.\end{cases}
$$

\begin{theorem}(\cite{garoufalidis2023quantum}).
For an oriented triangulated pseudo 3-manifold $X$ satisfying $\mathrm{H}_2(X, \mathbb{Z})=0$, the partition function
$$
Z_{F}(X):=\int_{z \in \mathsf{A}^{\Delta_3(X)}} K_X D_{X, \theta} \mathrm{d} z
$$
is invariant under shaped 2–3 Pachner moves. $Z_{F}(X)$ is the partition function of face-type Teichmüller TQFT over local field $F$.
\end{theorem}

\subsection{Two edge-type Teichmüller TQFT over local field}

Now we construct partition functions of a pair of edge-type generalized Teichmüller TQFTs over local field $F$ using the Weil transform of quantum dilogarithm, one with respect to $\hat{\mathsf{B}}$ and another with respect to $\mathsf{B}$. These Weil transforms lead to an edge-type state integral whose states are placed on the edges of the tetrahedra. The weights of the tetrahedra depend only on the combinatorial information of the triangulation, as encoded by the Neumann-Zagier matrices\cite{Neumann1985VolumesOH}.

Firstly we consider the $\hat{\mathsf{B}}$-Weil transform\cite{Garoufalidis2017AME}. Define the function
$$g_{a, c}(x, z)=g_{a, c}((\alpha, x),(\gamma, z)):=\alpha(z) \int_{\hat{\mathsf{B}}} \bar{\Psi}_{a, c}(-\alpha-\beta, 1 / x)\langle(\beta, 1) ;(\gamma, z)\rangle \mathrm{d} \beta.$$
Using (\ref{general charged qd}) we have
\begin{equation}
g_{a, c}(x, z)=f_{\dot{a}, \dot{c}}(1 /(x \ddot{a}), z \ddot{c}), \quad f_{\dot{a}, \dot{c}}(x, z):=\|x\|^{\dot{c}}\|z\|^{\dot{a}} \delta_F(x+z-1).    
\end{equation}
Define a $\mathsf{B}$-Boltzmann weight
$$B^{(\mathsf{B})}(T,x):=g_{a,c}(\frac{x_{0,2} x_{1,3}}{x_{0,3} x_{1,2}},\frac{x_{0,1} x_{2,3}}{x_{0,2} x_{1,3}}),$$
if $T$ is positive and conjugate otherwise. Here $x_{i, j}$ is the $\mathsf{B}$-valued state variable on the geometric edge opposite the edge $\partial_i \partial_j T$ of the tetrahedron $T$. Then the partition function of $\mathsf{B}$-type Teichmüller TQFT over local field $F$ is defined by

$$
Z_{F}^{(\mathsf{B})}(X):=\int_{\mathsf{B}^{\Delta_1(X)}}\left(\prod_{T \in \Delta_3(X)} B^{\mathsf{B})}\left(T,\left.x\right|_{\Delta_1(T)}\right)\right) d x
$$

Similarly one can define the partition function of $\hat{\mathsf{B}}$-type Teichmüller TQFT over local field $F$ based on $\hat{\mathsf{B}}$-Boltzmann weight
$$B^{(\hat{\mathsf{B}})}(T,\alpha):=h_{a,c}(\alpha_{0,2}+\alpha_{1,3}-\alpha_{0,3}-\alpha_{1,2},\alpha_{0,1}+\alpha_{2,3}-\alpha_{0,2}-\alpha_{1,3}),$$
where
$$
h_{a, c}(\alpha, \gamma)=h_{a, c}((\alpha, x),(\gamma, z)):=\gamma(x) \int_{\mathsf{B}} \bar{\Psi}_{a, c}(-\alpha, 1 /(x y))\langle(0, y) ;(\gamma, z)\rangle \mathrm{d} y .
$$

\section{Examples of $\mathsf{B}$-type}

\begin{example}
Consider the ideal triangulation of the complement of the $\mathbf{3}_1$ in $S^3$ with two positive tetrahedrons $T_1$ and $T_2$ with the edge identification
\begin{equation}
\begin{array}{c|cccccc|}
\hbox{\diagbox[width=2.5em]{\tiny{tet}}{\tiny{edge}}} &  01 & 02 & 03 & 12 & 13 & 23 \\
\hline
    1  &            s &  s &  t &  s &  s &  s \\
    2  &            s &  s &  t &  s &  s &  s \\
  \end{array}
  \end{equation}
Then we have
$$
x_1=\frac{s^2}{st}=\frac{s}{t}:=x,\quad z_1=\frac{s^2}{s^2}=1,\quad x_2=\frac{s}{t}=\frac{s}{t},\quad z_2=1.
$$
The balance condition on the edge is
$$
c_{1}+c_{2}=(2,1).
$$    
 The angle holonomy of half of the longitude is given by
 $$\lambda=2a_1-2a_2-c_1$$
and the angle holonomy of the meridian is given by
$$\mu=a_2-a_1.$$
Then the partition function of $\mathsf{B}$-type Teichmüller TQFT over local field $F$ associate to $S^3\backslash \mathbf{3}_1$ is given by
$$
\begin{aligned}
Z_{F}^{(\mathsf{B})}(S^3\backslash \mathbf{3}_1)=&\int_{\mathsf{B}} g_{a_1, c_1}\left(x, 1 \right) g_{a_2, c_2}\left(x, 1\right) \mathrm{d} x\\ 
=&\int_{\mathsf{B}} f_{\dot{a}_1, \dot{c}_1}\left(\frac{1}{x \ddot{a}_1}, \ddc_1\right) f_{\dot{a}_2, \dot{c}_2}\left(\frac{1}{x \ddot{a}_2}, \ddc_2\right) \mathrm{d} x\\
=&\int_{\mathsf{B}} f_{\dot{a}_1, \dot{c}_1}\left(\frac{1}{x}, \ddc_1\right) f_{\dot{a}_2, \dot{c}_2}\left(\frac{\ddot{a}_1}{x \ddot{a}_2}, \ddc_2\right) \mathrm{d} x\\
=&\int_{\mathsf{B}} f_{\dot{a}_1, \dot{c}_1}\left(\frac{1}{x},\frac{1}{\ddm^2\ddot{\lambda}}\right) f_{\dot{a}_2, \dot{c}_2}\left(\frac{1}{x\ddot{\mu}},\ddm^2\ddl\right) \mathrm{d} x \\
=&\int_{\mathsf{F}} \frac{\|\ddot{\lambda}\|^{-\da_1+\da_2}}{\|\ddot{\mu}\|^{2\da_1-2\da_2-\dc_1}\|x\|^{\dc_1+\dc_2}} \delta_{\mathsf{F}}\left(\frac{1}{x}+\frac{1}{\ddm^2\ddot{\lambda}}-1\right) \delta_{\mathsf{F}}\left(\frac{1}{x\ddot{\mu}}+\ddm^2\ddl-1\right) \frac{\mathrm{d} x}{\|x\|} \\
=&\int_\mathsf{F}\frac{\|\ddot{\lambda}\|^{\dot{\mu}}}{\|\ddot{\mu}\|^{\dot{\lambda}}\|x\|} \delta_{\mathsf{F}}\left(1+\frac{x}{\ddm^2\ddl}-x\right) \delta_{\mathsf{F}}\left(\frac{1}{\ddm}+\ddm^2\ddl x-x\right) \mathrm{d} x \\
=& \frac{\|\ddot{\lambda}\| ^{\dot{\mu}}\|\ddm\|\|1-\ddm^2\ddl\|}{\|\ddot{\mu}\|^{\dot{\lambda}}} \delta_{\mathsf{F}}\left(1-\ddm^2\ddl+\frac{1}{\ddm^3\ddl}-\frac{1}{\ddm}\right) \\
=& \frac{\|\ddot{\lambda}\| ^{\dot{\mu}}\|\ddm\|}{\|\ddot{\mu}\|^{\dot{\lambda}}} \delta_{\mathsf{F}}\left(1+\frac{1}{\ddm^3\ddl}\right) \\
=&\frac{\|\ddot{\lambda}\|^{\dot{\mu}}}{\|\ddot{\mu}\|^{\dot{\lambda}}}\|\ddot{\lambda}\|\|\ddot{\mu}\|^4\ \delta_{\mathsf{F}}\left(A_{\mathbf{3}_1}(\ddot{\lambda}, \ddot{\mu}^{\frac{1}{2}})\right) ,\\
\end{aligned}
$$    
where $$A_{\mathbf{3}_1}(l, m)=m^6 l+1.$$
\end{example}

\begin{example}
Consider the ideal triangulation of the complement of the $\mathbf{4}_1$ in $S^3$ with one positive tetrahedron $T_1$ and one negative tetrahedron $T_2$ with the edge identification
\begin{equation}
\begin{array}{c|cccccc|}
\hbox{\diagbox[width=2.5em]{\tiny{tet}}{\tiny{edge}}} &  01 & 02 & 03 & 12 & 13 & 23 \\
\hline
    1  &            s &  t &  s &  t &  t &  s \\
    2  &            t &  s &  t &  s &  s &  t \\
  \end{array}
  \end{equation}
Then we have
$$
x_1=\frac{t^2}{st}=\frac{t}{s}:=x,\quad z_1=\frac{s^2}{t^2},\quad x_2=\frac{s^2}{st}=\frac{s}{t},\quad z_2=\frac{t^2}{s^2}.
$$
The balance condition on the edge is
$$
2 a_{1}+c_{1}+2 b_{2}+c_{2}=(2,1).
$$    
 The angle holonomy of half of the longitude is given by
 $$\lambda=2a_1+c_1+\varpi$$
and the angle holonomy of the meridian is given by
$$\mu=a_1-a_2.$$
Then the partition function of $\mathsf{B}$-type Teichmüller TQFT over local field $F$ associate to $S^3\backslash \mathbf{4}_1$ is given by
$$
\begin{aligned}
Z_{F}^{(\mathsf{B})}(S^3\backslash \mathbf{4}_1)=&\int_{\mathsf{B}} g_{a_1, c_1}\left(x, 1 / x^2\right) \bar{g}_{a_2, c_2}\left(1 / x, x^2\right) \mathrm{d} x\\ 
=&\int_{\mathsf{B}} f_{\dot{a}_1, \dot{c}_1}\left(\frac{1}{x \ddot{a}_1}, \frac{\ddot{c}_1}{x^2}\right) f_{\dot{a}_2, \dot{c}_2}\left(x \ddot{a}_2, \frac{x^2}{\ddot{c}_2}\right) \mathrm{d} x\\
=&\int_{\mathsf{B}} f_{\dot{a}_1, \dot{c}_1}\left(\frac{1}{x}, \frac{\ddot{c}_1 \ddot{a}_1^2}{x^2}\right) f_{\dot{a}_2, \dot{c}_2}\left(\frac{x \ddot{a}_2}{\ddot{a}_1}, \frac{x^2}{\ddot{c}_2 \ddot{a}_1^2}\right) \mathrm{d} x\\
=&\int_{\mathsf{B}} f_{\dot{a}_1, \dot{c}_1}\left(\frac{1}{x},-\frac{\ddot{\lambda}}{x^2}\right) f_{\dot{a}_2, \dot{c}_2}\left(\frac{x}{\ddot{\mu}},-\frac{x^2}{\ddot{\lambda}\ddot{\mu}^2}\right) \mathrm{d} x \\
=&\int_{\mathsf{F}} \frac{\|\ddot{\lambda}\|^{\dot{a}_1-\dot{a}_2}\|x\|^{2 \dot{a}_2+\dot{c}_2}}{\|\ddot{\mu}\|^{2 \dot{a}_2+\dot{c}_2}\|x\|^{2 \dot{a}_1+\dot{c}_1}} \delta_{\mathsf{F}}\left(\frac{1}{x}-\frac{\ddot{\lambda}}{x^2}-1\right) \delta_{\mathsf{F}}\left(\frac{x}{\ddot{\mu}}-\frac{x^2}{\ddot{\lambda}\ddot{\mu}^2}-1\right) \frac{\mathrm{d} x}{\|x\|} \\
=&\int_\mathsf{F}\frac{\|\ddot{\lambda}\|^{\dot{\mu}}\|x\|}{\|\ddot{\mu}\|^{\dot{\lambda}-1}} \delta_{\mathsf{F}}\left(x-\ddot{\lambda}-x^2\right) \delta_{\mathsf{F}}\left(\frac{x}{\ddot{\mu}}-\frac{x^2}{\ddot{\lambda}\ddot{\mu}^2}-1\right) \mathrm{d} x \\
=&\int_\mathsf{F} \frac{\|\ddot{\lambda}\| ^{\dot{\mu}}\|\ddot{\mu}\|^2\|x\|}{\|\ddot{\mu}\|^{\dot{\lambda}}} \delta_{\mathsf{F}}\left(x-\ddot{\lambda}-x^2\right) \delta_{\mathsf{F}}\left(x-\frac{x-\ddot{\lambda}}{\ddot{\lambda} \ddot{\mu}}-\ddot{\mu}\right) \mathrm{d} x \\
=&\frac{\|\ddot{\lambda}\|^{\dot{\mu}}\|\ddot{\mu}\|^2\|\ddot{\mu}-1 / \ddot{\mu}\|}{\|\ddot{\mu}\|^{\dot{\lambda}}} \delta_{\mathsf{F}}\left((\ddot{\mu}-\ddot{\lambda})\left(1-\frac{1}{\ddot{\lambda} \ddot{\mu}}\right)-(\ddot{\mu}-1 / \ddot{\mu})^2\right)\\
=&\frac{\|\ddot{\lambda}\|^{\dot{\mu}}}{\|\ddot{\mu}\|^{\dot{\lambda}}}\|\ddot{\lambda}\|\|\ddot{\mu}\|^4\|\ddot{\mu}-1 / \ddot{\mu}\| \delta_{\mathsf{F}}\left(A_{\mathbf{4}_1}(-\ddot{\lambda}, \ddot{\mu}^{\frac{1}{2}})\right) ,\\
\end{aligned}
$$
where
$$A_{\mathbf{4}_1}(l, m)=  m^4 l^2-\left(m^8-m^6-2 m^4-m^2+1\right) l+m^4.$$
\end{example}

\begin{example}
Consider the ideal triangulation of the complement of the $\mathbf{5}_2$ in $S^3$ with three positive tetrahedron $T_j$ for $j=1,2,3$ with the edge identification given by
\begin{equation}
\begin{array}{c|cccccc|}
\hbox{\diagbox[width=2.5em]{\tiny{tet}}{\tiny{edge}}} &  01 & 02 & 03 & 12 & 13 & 23 \\
\hline
    1  &            s &  t &  t &  s &  u &  u \\
    2  &            u &  s &  t &  t &  t &  u \\
    3  &            s &  u &  t &  u &  s &  t
    \\
  \end{array}
  \end{equation}
Then we have
$$
\begin{aligned}
x_1=\frac{tu}{st}=\frac{u}{s}:=x,&\quad z_1=\frac{su}{tu}=\frac{s}{t}:=z,\\
x_2=\frac{st}{t^2}=\frac{s}{t}=z,&\quad z_2=\frac{u^2}{st}=x^2z,\\
x_3=\frac{su}{tu}=\frac{s}{t}=z,&\quad z_3=\frac{st}{su}=\frac{t}{u}=\frac{1}{xz}.\\    
\end{aligned}
$$
The balance condition on the two edges is
$$
\begin{aligned}
b_1+c_1+b_2+2 c_2+a_3+c_3&=(2,1),\\ a_1+b_1+2 a_2+b_3+c_3&=(2,1)\\    
\end{aligned}
$$    
 The angle holonomy of half of the longitude is given by
 $$\lambda=2 a_1+4 a_2-3 a_3-c_2 ,$$
and the angle holonomy of the meridian is given by
$$\mu=a_3-a_2.$$
Then the partition function of $\mathsf{B}$-type Teichmüller TQFT over local field $F$ associate to $S^3\backslash \mathbf{5}_2$ is given by
$$
\begin{aligned}
Z_{F}^{(\mathsf{B})}(S^3\backslash \mathbf{5}_2)=&\int_{\mathsf{B}^2} g_{a_1, c_1}\left(x, z\right) g_{a_2, c_2}\left(z, x^2z\right)g_{a_3, c_3}\left(z, \frac{1}{xz}\right) \mathrm{d} x\mathrm{d} z\\ 
=&\int_{\mathsf{B}^2} f_{\dot{a}_1, \dot{c}_1}\left(\frac{1}{x \ddot{a}_1}, z\ddc_1\right) f_{\dot{a}_2, \dot{c}_2}\left(\frac{1}{z\dda_2},x^2z\ddc_2\right)f_{\dot{a}_3, \dot{c}_3}\left(\frac{1}{z\dda_3},\frac{\ddc_3}{xz}\right) \mathrm{d} x\mathrm{d} z\\
=&\int_{\mathsf{B}^2} f_{\dot{a}_1, \dot{c}_1}\left(\frac{1}{x}, \frac{z\ddc_1}{\dda_2}\right) f_{\dot{a}_2, \dot{c}_2}\left(\frac{1}{z},\frac{x^2z\ddc_2}{\dda_1^2\dda_2}\right)f_{\dot{a}_3, \dot{c}_3}\left(\frac{\dda_2}{z\dda_3},\frac{\dda_1\dda_2\ddc_3}{xz}\right) \mathrm{d} x\mathrm{d} z\\
=&\int_{\mathsf{B}^2} f_{\dot{a}_1, \dot{c}_1}\left(\frac{1}{x}, \frac{z}{\ddm}\right) f_{\dot{a}_2, \dot{c}_2}\left(\frac{1}{z},\frac{x^2z}{\ddl\ddm^3}\right)f_{\dot{a}_3, \dot{c}_3}\left(\frac{1}{z\ddm},\frac{\ddl\ddm^2}{xz}\right) \mathrm{d} x\mathrm{d} z\\
=&\int_{\mathsf{F}^2} \|\ddl\|^{-\da_2+\da_3}\|\ddm\|^{-\da_1-3\da_2-\dc_3+2\da_3}\|x\|^{2\da_2-\da_3-\dc_1}\|z\|^{\da_1+\da_2-\da_3-\dc_2-\dc_3}\\
&\times\delta_{\mathsf{F}}\left(\frac{1}{x}+\frac{z}{\ddm}-1\right) \delta_{\mathsf{F}}\left(\frac{1}{z}+\frac{x^2z}{\ddl\ddm^3}-1\right) \delta_{\mathsf{F}}\left(\frac{1}{z\ddm}+\frac{\ddl\ddm^2}{xz}-1\right) \frac{\mathrm{d} x}{\|x\|} \frac{\mathrm{d} z}{\|z\|}\\
=&\int_{\mathsf{F}^2}\frac{\|\ddot{\lambda}\|^{\dot{\mu}}}{\|\ddot{\mu}\|^{\dot{\lambda}}} \delta_{\mathsf{F}}\left(1+\frac{xz}{\ddm}-x\right) \delta_{\mathsf{F}}\left(\frac{1}{z}+\frac{x^2z}{\ddl\ddm^3}-1\right) \delta_{\mathsf{F}}\left(\frac{1}{\ddm}+\frac{\ddl\ddm^2}{x}-z\right)  \mathrm{d} x  \mathrm{d} z\\
=&\int_{\mathsf{F}}\frac{\|\ddot{\lambda}\|^{\dot{\mu}}}{\|\ddot{\mu}\|^{\dot{\lambda}}} \delta_{\mathsf{F}}\left(1+\frac{x(\frac{1}{\ddm}+\frac{\ddl\ddm^2}{x})}{\ddm}-x\right) \delta_{\mathsf{F}}\left(\frac{1}{\frac{1}{\ddm}+\frac{\ddl\ddm^2}{x}}+\frac{x^2(\frac{1}{\ddm}+\frac{\ddl\ddm^2}{x})}{\ddl\ddm^3}-1\right)  \mathrm{d} x\\
=&\int_{\mathsf{F}}\frac{\|\ddot{\lambda}\|^{\dot{\mu}}}{\|\ddot{\mu}\|^{\dot{\lambda}}\|x\|} \delta_{\mathsf{F}}\left((\frac{1}{\ddm^2}-1)x+\ddl\ddm+1\right) \delta_{\mathsf{F}}\left(\frac{\ddm}{x+\ddl\ddm^3}+\frac{x}{\ddl\ddm^4}+\frac{1}{\ddm}-\frac{1}{x}\right)  \mathrm{d} x\\
=&\frac{\|\ddot{\lambda}\|^{\dot{\mu}}}{\|\ddot{\mu}\|^{\dot{\lambda}}}\frac{1}{\|\frac{1}{\ddm^2}-1\|\|\frac{\ddl\ddm+1}{1-\frac{1}{\ddm^2}}\|}\delta_{\mathsf{F}}\left(\frac{\ddm}{(\frac{\ddl\ddm+1}{1-\frac{1}{\ddm^2}})+\ddl\ddm^3}+\frac{\frac{\ddl\ddm+1}{1-\frac{1}{\ddm^2}}}{\ddl\ddm^4}+\frac{1}{\ddm}-\frac{1}{(\frac{\ddl\ddm+1}{1-\frac{1}{\ddm^2}})}\right)\\
=&\frac{\|\ddot{\lambda}\|^{\dot{\mu}}}{\|\ddot{\mu}\|^{\dot{\lambda}}}\|\ddot{\lambda}\|\|\ddot{\mu}\|^4\|\ddot{\mu}-1 / \ddot{\mu}\| \|1+\ddl\ddm^3\|\delta_{\mathsf{F}}\left(A_{\mathbf{5}_2}(\ddot{\lambda}, \ddot{\mu}^{\frac{1}{2}})\right), \\
\end{aligned}
$$   
where 
$$\begin{aligned}
A_{5_2}(l, m)= & l^3+\left(m^{10}-m^8+2 m^4+2 m^2-1\right) l^2 \\
& -m^4\left(m^{10}-2 m^8-2 m^6+m^2-1\right) l+m^{14} .
\end{aligned}$$
\end{example}

\section{A theorem connecting A-polynomial and state-integral}\label{section 3}

After observing the above examples, we find that for knot complements the form of the partition function of $\mathsf{B}$-type Teichmüller TQFT should be 
$$
Z_F^{(\mathsf{B})}(S^3\backslash K)=\frac{\|\ddot{\lambda}\|^{\dot{\mu}}}{\|\ddot{\mu}\|^{\dot{\lambda}}}\|\mathcal{F}_K(\ddl,\ddm)\|\delta_{\mathsf{F}}\left(A_{K}(\ddot{\lambda}, \ddot{\mu}^{\frac{1}{2}})\right),
$$
for proper choice of angle holonomies of half of longitude and meridian. 

Let $M$ be an oriented one-cusped 3-manifold equipped with an oriented ideal triangulation $\mathcal{T}=\left\{\Delta_j\right\}_{j=1}^N$. Recall that to each tetrahedron $\Delta_j$ we can associate three variables $z_j, z_j^{\prime}, z_j^{\prime \prime}$, satisfying the following equations
\begin{equation}\label{hyperbolic eq 1}
z_j z_j^{\prime} z_j^{\prime \prime}=-1    
\end{equation}
and
\begin{equation}\label{hyperbolic eq 2}
z_j+\left(z_j^{\prime}\right)^{-1}-1=0.    
\end{equation}

Recall the gluing equation
\begin{equation}\label{gluing equation z}
\prod_{j=1}^n z_j^{a{i j}}\left(z_j^{\prime}\right)^{b}_{i j}\left(z_j^{\prime \prime}\right)^{c_{i j}}=1,\quad i=1,\dots,N, 
\end{equation}
and the cusp equation
\begin{equation}\label{cusp equation}
    m=\prod_{j=1}^n z_j^{a_{ j}^{m}}(z')_j^{b_{ j}^{m}}(z'')_j^{c_{ j}^{m}}, \quad l=\prod_{j=1}^n z_j^{a_{ j}^{l}}(z'_j)^{b_{ j}^{l}}(z''_j)^{c_{ j}^{l}},
\end{equation}
where $l$ and $m$ correspond to topological longitude and meridian. 
We orient $m$ and $l$ that the tangent vectors $v_{m}$ and $v_{l}$ to $m$ and $l$, respectively, at the point where $m$ intersects $l$ are oriented according to the right-hand rule.

The equations (\ref{hyperbolic eq 1}) and (\ref{hyperbolic eq 2}) relating $z, z^{\prime}, z^{\prime \prime}$ variables, the gluing equations (\ref{gluing equation z}), and the cusp equations (\ref{cusp equation}) are equations in the variables $z_j, z_j^{\prime}, z_j^{\prime \prime}$ and $l, m$. Solving these equations for $l, m$, eliminating the variables $z_j, z_j^{\prime}, z_j^{\prime \prime}$, we obtain a relation between $l$ and $m$.

\begin{theorem}\label{champanerkar theorem}
(\cite{champanerkar2003polynomial}). 
When we solve the system of equations (\ref{hyperbolic eq 1}), (\ref{hyperbolic eq 2}), (\ref{gluing equation z}) and (\ref{cusp equation}) in terms of $m$ and $\ell$, we obtain a factor of the  $\operatorname{PSL}(2, \mathbb{C})$ A-polynomial.
\end{theorem}

Note that the equation $\frac{1}{x\dda}+z\ddc-1=0$ is equivalent to 
$$
\ddb^{-1}x_{0,1} x_{2,3}+\dda x_{0,2} x_{1,3}-x_{0,3} x_{1,2}=0
$$
when $x=\frac{x_{0,2} x_{1,3}}{x_{0,3} x_{1,2}}$ and $z=\frac{x_{0,1} x_{2,3}}{x_{0,2} x_{1,3}}$.
We now prove the following theorem, which explains why the A-polynomial appears in the delta function.

\begin{theorem}
Fix an ideal triangulation $\mathcal{T}$ of an oriented one-cusped 3-manifold $X$ with $N$ tetrahedra $T_i$ for $i=1, \ldots, N$. Choose a $\mathsf{C}$-valued shape structure $\theta$ on $X$. Assign variables $x_k$ for $k=1, \ldots, N$ to $N$ edges of $\mathcal{T}$. Consider the following $N$ equations
\begin{equation}\label{old P equation}
E_{T_j}:\begin{cases}\ddb_j^{-1}X_{\alpha}^{(j)}+\dda_jX_{\beta}^{(j)}-X_{\gamma}^{(j)}=0& \text { if } \operatorname{sgn}(T_j)=+1, \\\ddb_jX_{\alpha}^{(j)}+\dda_j^{-1}X_{\beta}^{(j)}-X_{\gamma}^{(j)}=0& \text { if } \operatorname{sgn}(T_j)=-1,\end{cases}\quad j=1,\dots,N,
\end{equation}
where $X_{\alpha}^{(j)}=x_{0,1}^{(j)}x_{2,3}^{(j)}$, $X_{\beta}^{(j)}=x_{0,2}^{(j)}x_{1,3}^{(j)}$, $X_{\alpha}^{(j)}=x_{0,3}^{(j)}x_{1,2}^{(j)}$ 
and $x_{m,l}^{(j)}$ is the variable on the edge opposite to the edge $\partial_m \partial_l T_j$. 

With the above assumptions, there exists a column vector $\eta$ whose coordinates depend only on $\ddot{\theta}$, such that after change of variable $x_k=\eta_k y_k$ for $k=1,\dots,N$, the equations (\ref{old P equation}) become
\begin{equation}\label{new P equation}
E'_{T_j}:\begin{cases}\dda_j+\ddb_j^{-1}=0& \text { if } \operatorname{sgn}(T_j)=+1, \\\ddb_j+\dda_j^{-1}-1=0& \text { if } \operatorname{sgn}(T_j)=-1,\end{cases}\quad j=1,\dots,N.
\end{equation}

Using the balance condition on edges and angle holonomies $\lambda$, $\mu$ of longitude and meridian, we can solve the equations (\ref{new P equation}) in terms of $\ddl$ and $\ddm$ and obtain a factor of the $\operatorname{PSL}(2, \mathbb{C})$ A-polynomial.
\end{theorem}

\begin{proof}
Firstly, we consider an ideal triangulation with $N$ positive tetrahedra. Combining the balance condition and the angle holonomy corresponding to the meridian, we can obtain the following $N$ equations
\begin{equation}
    \prod_{j=1}^N \dda_j^{a_{ij}}\ddb_j^{b_{ij}}\ddc_j^{c_{ij}}=1,\quad j=1,\dots,N.
\end{equation}
Set $A:=(a_{ij})$, $B:=(b_{ij})$, $C:=(c_{ij})$, $A':=A-C$ and $B':=B-C$. Using the theory of Neumann-Zagier\cite{Neumann1985VolumesOH} we can assume that $A'$ and $B'$ satisfy 
\begin{enumerate}
    \item $(A'\quad B')$ is the upper half of a symplectic $2 N \times 2 N$ matrix;
    \item there exists a quad such that $B'$ is invertible\cite[Lem.A.3]{Dimofte:2012qj}.
\end{enumerate}

Let $e_i \in \mathbb{Z}^N$ denote the $i$-th coordinate vector. We use the shorthand notation $\vec{z}^{\vec{v}}:=\prod^{N}_{i=1}z_i v_i$. One can see that
$$
X_{\alpha}^{(j)}=\vec{x}^{Ae_j},\quad 
X_{\beta}^{(j)}=\vec{x}^{Be_j},\quad 
X_{\gamma}^{(j)}=\vec{x}^{Ce_j},
$$
where $\vec{x}=(x_1,\dots,x_N)$. This is because matrices $A$, $B$, and $C$, indexed by edges and tetrahedra, record the count of each shape of the $j$-th tetrahedron appearing around the $i$-th edge.

Denote by $\vec{A}$, $\vec{B}$ and $\vec{E}$ the formal logarithm of $\vec{\dda}=(\dda_1,\dots,\dda_N)$, $\vec{\ddb}=(\ddb_1,\dots,\ddb_N)$ and $\vec{\eta}=(\eta_1,\dots,\eta_N)$. Set $\vec{y}=(y_1,\dots, y_N)$. Substituting $\vec{x}=\vec{y}^{\vec{\eta}}$ into equations (\ref{old P equation}) we obtain
\begin{equation}
\begin{aligned}
    &e^{-B_j+e_j^{T}(A-C)^{T}\vec{E}}\vec{y}^{(A-C)e_j}+e^{A_j+e_j^{T}(B-C)^{T}\vec{E}}\vec{y}^{(B-C)e_j}-1\\
    =&e^{-B_j+e_j^{T}A'^{T}\vec{E}}\vec{y}^{A'e_j}+e^{A_j+e_j^{T}B'^{T}\vec{E}}\vec{y}^{B'e_j}-1=0.\quad j=1,\dots,N.\\
\end{aligned}
\end{equation}
We want to find a vector $\vec{E}$, which is dependent only on $\ddot{\theta}$, such that
\begin{equation}\label{the equation of vector E}
\begin{aligned}
        -\vec{B}+A'^{T}\vec{E}=&\vec{P},\\
        \vec{A}+B'^{T}\vec{E}=&\vec{Q},\\
\end{aligned}
\end{equation}
where $e^{P_j}=\ddb_j^{-1}\vec{y}^{-A'e_j}$ and $e^{Q_j}=\dda_j\vec{y}^{-B'e_j}$. Note the fact that $\vec{P}$ and $\vec{Q}$ are dependent only on $\ddl$ and $\ddm$ following from the argument in \cite[Prop.2.53]{howie2021apolynomials} for the variables $y_k$. 

The solution of the second equation of (\ref{the equation of vector E}) is 
\begin{equation}\label{the solution of vector E}
 \vec{E}=(B'^{T})^{-1}(\vec{Q}-\vec{A}).   
\end{equation}
Substituting (\ref{the solution of vector E}) into the first equation of (\ref{the equation of vector E}) we obtain
$$
\begin{aligned}
    &-\vec{B}+A'^{T}((B'^{T})^{-1}(\vec{Q}-\vec{A}))\\
    =&B'^{-1}(C\mathbf{1}-\mathbf{2})+B'^{-1}A'\vec{A}+A'^{T}(B'^{T})^{-1}\vec{Q}-A'^{T}(B'^{T})^{-1}\vec{A}\\
    =&B'^{-1}(C\mathbf{1}-\mathbf{2})+B'^{-1}A'\vec{A}+B'^{-1}A'\vec{Q}-B'^{-1}A'\vec{A}\\
    =&B'^{-1}(C\mathbf{1}-\mathbf{2}+A'\vec{Q})\\
    =&B'^{-1}B'\vec{P}=\vec{P}.
\end{aligned}
$$
The first equality follows from the fact that
$$
A'\vec{A}+B'\vec{B}=\mathbf{2}-C\mathbf{1}.
$$
The second equality follows from the fact that $(A'\quad B')$ is the upper half of a symplectic matrix, hence $B'^{-1}A'$ is symmetric. The second-to-last equality follows from
$$
\begin{aligned}
&\prod^N_{j=1}(\dda_j\frac{Y_\gamma^{(j)}}{Y_\beta^{(j)}})^{A'_{ij}}(\ddb_j\frac{Y_\alpha^{(j)}}{Y_\gamma^{(j)}})^{B'_{ij}}\\ 
=&\prod^{N}_{j=1}(\frac{Y_\gamma^{(j)}}{Y_\beta^{(j)}})^{a_{ij}}(\frac{Y_\alpha^{(j)}}{Y_\gamma^{(j)}})^{b_{ij}}(\frac{Y_\beta^{(j)}}{Y_\alpha^{(j)}})^{c_{ij}}\prod^{N}_{j=1}\dda_j^{A'_{ij}}\ddb_j^{B'_{ij}}\\
=&\prod^{N}_{j=1}\prod^{N}_{i=1}y^{(c_{ij}-b_{ij})a_{ij}+(a_{ij}-c_{ij})b_{ij}+(b_{ij}-a_{ij})c_{ij}}\prod^{N}_{j=1}\dda_j^{A'_{ij}}\ddb_j^{B'_{ij}}\\
=&\prod^{N}_{j=1}\dda_j^{A'_{ij}}\ddb_j^{B'_{ij}}=e^{(\mathbf{2}-C\mathbf{1})_i}.\\
\end{aligned}
$$

Thus, after set $\vec{x}=\vec{y} ^{\vec{\eta}}$, where $\vec{\eta}=e^{\vec{E}}$ and $E$ is defined as (\ref{the solution of vector E}), the equations (\ref{old P equation}) transform into the equations (\ref{new P equation}). Now we have four equations for $\ddot{\theta}$:
\begin{enumerate}
    \item $\dda\ddb\ddc=-1$;
    \item $\dda+\ddb^{-1}-1=0$;
    \item balanced condition;
    \item cusped holonnomies for $\ddl$ and $\ddm$.
\end{enumerate}

Finally, Theorem \ref{champanerkar theorem} shows that when we solve the above equations in terms of $\ddl$ and $\ddm$, we obtain a divisor of the $\operatorname{PSL}(2, \mathbb{C})$ A-polynomial of $X$. 

For the case of negative tetrahedra, swap the positions of $a$ and $b$, which leads to the second equation of (\ref{new P equation}). This completes the proof.
\end{proof}

\begin{remark}
Observe that the function $\mathcal{F}_K$ has an interesting form
$$
\begin{aligned}
    \mathcal{F}_{\mathbf{3}_1}(x,y)=&\|xy^4\|,\\
    \mathcal{F}_{\mathbf{4}_1}(x,y)=&\|xy^4(y-y^{-1})\|,\\
    \mathcal{F}_{\mathbf{5}_2}(x,y)=&\|xy^4(y-y^{-1})(1+xy^3)\|.\\
\end{aligned}
$$
\end{remark}
We will further investigate these new knot invariants in future studies.
\printbibliography

\end{document}